\documentclass[11pt,letterpaper]{amsart}
\usepackage{amsmath}
\usepackage{amssymb}

\usepackage{epsfig}
\usepackage[all, knot]{xy}
\xyoption{arc}

\newtheorem{thm}{Theorem}[section]
\newtheorem*{thmA}{Theorem~A}
\newtheorem{df}[thm]{Definition}
\newtheorem{prop}[thm]{Proposition}
\newtheorem{cor}[thm]{Corollary}

\newtheorem{lem}[thm]{Lemma}
\newtheorem{ex}[thm]{Example}
\newtheorem{rem}[thm]{Remark}
\newtheorem{question}[thm]{Question}

\newcommand{\Pic}{\operatorname{Pic}}
\newcommand{\Div}{\operatorname{Div}}
\newcommand{\Alb}{\operatorname{Alb}}

\newcommand{\pp}{\mathbb{P}}

\newcommand{\NS}{\operatorname{NS}}

\newcommand{\Ht}{\operatorname{Ht}}
\newcommand{\Alba}{\operatorname{Alb}}
\newcommand{\EFC}{\operatorname{Eff}}

\newcommand{\Flim}{\operatorname{Flim}}

\begin{document}

\title[Numerical equivalence for height functions]{The numerical equivalence relation for height functions\\and ampleness and nefness criteria for divisors}

\author{Chong Gyu Lee}


\date{\today}


\address{Department of Mathematics, University of Illinois at Chicago, IL, 60607 US}

\email{phiel@math.uic.edu}

\maketitle

\begin{abstract}
 In this paper, we study properties of Weil height functions associated with numerically trivial divisors.
 It helps us to define the fractional limit of $h_E$ with respect to $h_D$ on $U$, with $D$ ample:
 \[
        \Flim_D(E,U) := \liminf_{\substack{P \in U \\  h_D(P) \rightarrow \infty}}\dfrac{h_E(P)}{h_D(P)}.
 \]
 The value of $\Flim_D(E,U)$ contains numerical information about a divisor $E$, enough to determine whether $E$ is ample, numerically effective or pseudo-effective.
\end{abstract}

\section{Introduction}

        Let $V$ be a projective variety and let $E$ be a divisor on $V$. In this paper, we determine geometric properties of a divisor $E$ using a Weil height function $h_E$ associated with $E$.
        \begin{thmA}\label{main}
        Let $V$ be a projective variety defined over a number field $k$, let $U$ be a subset of $V$ and let $E,D$ be divisors on $V$ with $D$ ample. We define \emph{the fractional limit of $h_E$ with respect to $h_D$ on $U$} to be
        \[
        \Flim_D(E,U) := \liminf_{\substack{P \in U \\  h_D(P) \rightarrow \infty}}\dfrac{h_E(P)}{h_D(P)}.
        \]
        The fractional limit satisfies the following properties:
        \begin{enumerate}
        \item $E$ is ample if and only if  $\Flim_D(E,V) >0.$
        \item $E$ is numerically effective if and only if  $\Flim_D(E,V)  \geq 0.$
        \item $E$ is effective only if
        $\Flim_D(E,U)  \geq 0$
        for some dense open set $U \subset V$.
        \item $E$ is pseudo-effective only if
        $\Flim_D(E,U)  \geq 0$
        for some $U$ which is an infinite intersection of open subsets of $V$.
        \item $E$ is pseudo-effective if
        $\Flim_D(E,U)  \geq 0$
        for some dense open set $U \subset V$.
        \item Suppose that any pseudo-effective divisor on $V$ is linearly equivalent to a sum of an effective divisor and a numerically effective divisor. Then, $E$ is pseudo-effective if only if
        $\Flim_D(E,U) \geq 0$
        for some dense open set $U \subset V$.
        \end{enumerate}
    \end{thmA}
    \noindent Note that a pseudo-effective divisor is a limit of a sequence of effective divisors. (See Definition~\ref{pse}.)

    A fundamental tenet of diophantine geometry is that the geometric properties of an algebraic variety should determine its basic arithmetic properties. We can see a good example in Weil's work: Weil \cite{W} introduced one of the main ingredients of diophantine geometry, the Weil height machine. It is a group homomorphism between the Picard group of a projective variety $V$ defined over a number field $k$ and the group of equivalence classes of real valued functions on $V\bigl(\overline{k}\bigr)$,
     \[
     h: \Pic(V) \rightarrow \{f:V\bigl(\overline{k}\bigr) \rightarrow \mathbb{R}\}/\sim
     \]
     sending $E$ to a Weil height function $h_E$ associated with a divisor $E$ on $V$. (See Definition~2.7 for the equivalence relation of real valued functions.) It is known that
    \begin{center}
    \begin{tabular}{lll}
    $E$ is torsion & $\Longrightarrow$ & $|h_E|$ is bounded, \\
    $E$ is ample & $\Longrightarrow$ & $h_E$ is bounded below on $V$,  \\
    $E$ is effective & $\Longrightarrow$ & $h_E$ is bounded below on some dense open set on $U\subset V$.
    \end{tabular}
    \end{center}

    Conversely, it is also possible to extract geometric information about $E$ from arithmetic properties of $h_E$. For example, let $C$ be a projective curve and let $E_C$, $D_C$ be divisors on $C$ with $\deg (D_C) \geq 1$. Then, it is known that
    \[
    \lim_{h_{C,D_C}(P) \rightarrow \infty} \dfrac{h_{C,E_C}(P)}{h_{C,D_C}(P)} = \dfrac{\deg E_C}{\deg D_C}.
    \]
    (See \cite[Theorem~B.3.5]{SH}.) Thus, $E_C$ is ample (respectively numerically effective) if the above limit is positive (respectively nonnegative).

    Theorem~A is a generalization of this result to higher dimensions. Because of the relation between the intersection theory and Weil height functions, $\Flim_D(E,U)$ looks like a lower bound for $D\cdot C$
    for all effective curves $C$ on $V$ which properly intersect with $V\setminus U$. Indeed, let $V$ be a projective variety defined over a number field $k$, let $E$ be a divisor on $V$ and let $W$ be a subvariety of $V$ of dimension $m$. We can consider $\left. h_E \right|_W$ in following ways: let $\iota :W\to V$ be a closed embedding. Then, we get
    \[
    h_E\bigl( \iota(P) \bigr) = h_{\iota^*E}(P) +O(1).
    \]
    By intersection theory, $\iota^*D$ is a divisor on $W$, isomorphic to a $(m-1)$-cycle $W \cdot E$ on $V$. In particular, if we choose an irreducible curve $C$ on $V$ such that $C \setminus U$ is a finite set of points, then we get
    \[
    \dfrac{E \cdot C}{D \cdot C} = \lim_{\substack{ P\in C \\ h_{V,D}(P) \rightarrow \infty}} \dfrac{h_{V,E}(P)|_C}{h_{V,D}(P)|_C}
    \geq \liminf_{\substack{ P\in U \\ h_{V,D}(P) \rightarrow \infty}} \dfrac{h_{V,E}(P)}{h_{V,D}(P)}.
    \]
    Therefore, the fractional limit provides numerical information enough to determine whether $E$ is ample, numerically effective or pseudo-effective.

     If $E_1$ and $E_2$ generate equivalent Weil height functions, their fractional limits are the same. Also, by the algebraic equivalence property of the Weil height machine (Theorem~\ref{WHM}~(6)), $\Flim_D(E,U)$ depends only on the algebraic equivalence class of $E$. So, we need to study the kernel of the Weil height machine and properties of Weil height functions associated with algebraically trivial divisors. In Section~3, we show that the kernel of the Weil height machine is the torsion subgroup of $\Pic(V)$ using the Albanese variety (see \S2.4 for details on Albanese varieties). It helps us to study properties of Weil height functions associated with algebraically or numerically trivial divisors. In Section~4, we define the algebraic and the numerical equivalence of Weil height functions to show that the fractional limit actually depends on the numerical equivalence class of $E$. In Section~5, we show some key lemmas which complete the proof of Theorem~A.

\par\noindent\emph{Acknowledgements}.\enspace
 I would like to thank Joseph H. Silverman for his overall advice. I also thanks Dan Abramovich for his helpful instruction in algebraic geometry, especially for background on movable curves.

\section{preliminaries}

    In this section, we introduce several algebro-geometric definitions and theorems required later. Generally, we will adopt notations from \cite{SH} unless stated otherwise.

    \subsection{Divisors and Weil height functions on projective varieties}
    \begin{df}\label{at}
        Let $D_1,D_2$ be divisors on a projective variety $V$. We say \emph{$D_1$ is algebraically equivalent to $D_2$} if
        there exist a connected algebraic set $T$, two points $t_1, t_2 \in T$ and a divisor $\mathfrak{D}$ on $V \times T$ such that
        \[
        D_i = \mathfrak{D}|_{V \times \{t_i\}}
        \]
        for $i=1,2$. And we denote it by $D_1 \equiv D_2$. Also, we say \emph{$D_1$ is numerically equivalent to $D_2$} if
        \[
        D_1 \cdot C = D_2 \cdot C
        \]
        for all irreducible curves $C$ on $V$ and write $D_1 \equiv_n D_2$. For notational convenience, we say $D$ is \emph{algebraically trivial (respectively numerically trivial)} if $D \equiv 0$ {\rm (}respectively $D \equiv_n 0${\rm )}.
    \end{df}

    \begin{df}
        Let $V$ be a projective variety. We define \emph{the N\'{e}ron-Severi Group of $V$} to be the group of algebraic equivalence classes of divisors on $V$:
        \[
        \NS(V) = \Div(V) / \equiv.
        \]
        Also, we define \emph{the numerical N\'{e}ron-Severi Group of $V$} to be the group of numerical equivalence classes of divisors on $V$:
        \[
        N^1(V) = \Div(V) / \equiv_n.
        \]
    \end{df}

    \begin{rem}\label{NS}
    $\NS(V), N^1(V)$ are abelian groups of finite rank. More precisely, $N^1(V)$ is the free part of $\NS(V)$.
    \end{rem}

     \begin{df}
        Let $V$ be a smooth projective variety and let $D$ be a divisor on $V$. We say that \emph{$D$ is very ample} if there is a closed embedding $\iota:V\rightarrow \pp^N$ such that
        \[
        D = \iota^*H \quad \text{where}~H~\text{is a hyperplane of}~\pp^N.
        \]
        We say that \emph{$D$ is ample} if $mD$ is very ample for some positive integer $m$.
     \end{df}

     \begin{df}
        Let $h$ be the logarithmic absolute height function on $\pp^n\bigl(\overline{\mathbb{Q}}\bigr)$, let $V$ be a smooth projective variety defined over a number field $k$ and let $D$ be a very ample divisor on $V$. Then, we define \emph{a Weil height function on $V$ associated with $D$} to be
        \[
        h_{V,D}(P) := h \bigl( \iota_D(P) \bigr)
        \]
        where $\iota_D:V \rightarrow \pp^n$ is the closed embedding defining given very ample divisor $D$.
     \end{df}

   \begin{thm}[The Weil height machine]\label{WHM}
    Let $V$ be a smooth projective variety defined over a number field $k$. Then there exists a map between $\Div(V)$, the group of divisors on $V$ and the group of real valued functions on $V$
    \[
    h_V:\Div(V) \longrightarrow \{ functions~:V\bigl(\overline{k}\bigr) \rightarrow \mathbb{R}\}
    \]
    with the following properties:
    \begin{enumerate}
        \item {\rm (Normalization)} Let $H\subset \pp^n$ be a hyperplane and let $h$ be the absolute logarithmic height on $\pp^n\bigl(\overline{k}\bigr)$. Then,
            \[
            h_{\pp^n,H}(P) = h(P)+O(1) \quad \text{for all}~P\in \pp^n\bigl(\overline{k}\bigr).
            \]
        \item{\rm(Functoriality)} Let $\phi : V \rightarrow W$ be a morphism and let $D \in \Div(W)$. Then,
        \[
        h_{V, \phi^*D}(P) = h_{W,D}\bigl( \phi(P) \bigr) + O(1) \quad \text{for all}~P \in \bigl(\overline{k}\bigr).
        \]
        \item {\rm (Additivity)} Let $D,E \in \Div(V)$. Then,
        \[
        h_{V,D+E}(P) = h_{V,D}(P) + h_{V,E}(P) + O(1)  \quad \text{for all}~P\in V\bigl(\overline{k}\bigr).
        \]
        \item {\rm (Linear Equivalence)} Let $D,E \in \Div(V)$ with $D$ linearly equivalent to $E$. Then,
        \[
        h_{V,D}(P) = h_{V,E}(P) + O(1)  \quad \text{for all}~P\in V\bigl(\overline{k}\bigr).
        \]
        \item {\rm (Positivity)} Let $D \in \Div(V)$ be an effective divisor and let $B$ be the base locus of the linear system $|D|$. Then,
        \[
        h_{V,D}(P) \geq O(1)  \quad \text{for all}~P\in (V\setminus B)\bigl(\overline{k}\bigr).
        \]
        \item {\rm (Algebraic Equivalence)} Let $D,E \in \Div(V)$ with $D$ ample and $E$ algebraically trivial. Then,
        \[
        \lim_{\substack{P\in V\bigl(\overline{k}\bigr) \\ h_{V,D}(P) \rightarrow \infty } } \dfrac{h_{V,E}(P)}{h_{V,D}(P)} =0.
        \]
        \item {\rm (Finiteness)} Let $D \in \Div(V)$ be an ample divisor. Then for every finite extension  $k'/k$ and every constant $C$,
        the set
        \[
        \{ P \in V(k') ~|~ h_{V,D}(P) \leq C \}
        \]
        is finite.
        \item {\rm (Uniqueness)} The height functions $h_{V,D}$ are uniquely determined up to $O(1)$ by {\rm (1) Normalization, (2) Functoriality} and {\rm (3) Additivity}.
    \end{enumerate}
    \end{thm}
    \begin{proof}
    See \cite[Theorem B.3.2]{SH}.
    \end{proof}

    \begin{df}
        Let $V$ be a projective variety defined over a number field $k$ and let $h_1,h_2: V \rightarrow \mathbb{R}$ be real valued functions on $V$. We say
        \emph{$h_1$ is equivalent to $h_2$} if
        \[
        h_1(P) = h_2(P) + O(1)
        \]
        and write
        \[
        h_1 \sim h_2.
        \]
        We also define \emph{the group of equivalence classes of Weil height functions on $V$} to be
        \[
        \Ht\bigl( V(\overline{k})\bigr) := \{ h_{V,D} ~|~ D \in \Div(V\bigl(\overline{k}\bigr)) \}/\sim.
        \]
    \end{df}

    \subsection{Ampleness, numerical effectiveness and pseudo-effectiveness}

    We refer \cite{De} to the reader for details. We will assume that $V$ is a projective variety and $D,E$ are divisors on $V$ for notational convenience in this subsection.

    \begin{thm}[Nakai-Moishezon Criterion]\label{NMC}
    A divisor $D$ is ample if and only if
    \[
    D^r \cdot Y >0 \quad \text{where}~r~\text{ is the dimension of}~Y
    \]
    for all integral subvariety $Y$ of $V$.
    \end{thm}

    \begin{thm}[Kleiman's Criterion]\label{KC}
    A divisor $D$ is ample if and only if
    \[
    D \cdot C >0
    \]
    for all $C$ in closure the cone of effective curves on $V$.
    \end{thm}
    \begin{proof}
     See \cite{Kle} or \cite[Theorem 1.27]{De}.
    \end{proof}

    \begin{lem}\label{ampledifference}
        Let $V$ be a projective variety and let $D, E$ be ample divisors on $V$. Then, there is a positive real number $m>0$ such that
        $mD - E$ is ample again.
    \end{lem}
    \begin{proof}
    See \cite[Theorem A.3.2.3]{SH}.
    \end{proof}

    \begin{df}[Numerical effectiveness]
    A divisor $D$ is called \emph{numerically effective} if
    \[
    D^r \cdot Y \geq 0 \quad \text{where}~r~\text{ is the dimension of}~Y
    \]
    for all integral subvariety $Y$ of $V$.
    \end{df}

    \begin{thm}
    A divisor $D$ is numerically effective if and only if
    \[
    D \cdot C \geq 0
    \]
    for all irreducible curves $C$ on $V$.
    \end{thm}
    \begin{proof}
        See \cite[Theorem 1.26]{De}.
    \end{proof}

    \begin{df}\label{pse}
    We define \emph{the effective cone of $V$} to be
    \[
    \EFC(V) = \langle E ~|~ E~\text{is an effective divisor}\rangle \otimes \mathbb{R}.
    \]
    And, we call a divisor $E$ \emph{pseudo-effective} if $E \in \overline{\EFC(V)}$.
    \end{df}

    \begin{df}
        A curve $C\subset V$ is called \emph{movable} if there exists an irreducible algebraic family of curves
        \[
        \{C_s ~|~ s \in S\}
        \]
        containing $C$ as a reduced member and dominating $V$.
    \end{df}

    \begin{thm}[Boucksom-Demailly-Paun-Peternell]\label{movablecurve}
    Let $V$ be a projective variety and let $E$ be a divisor on $V$. Then, $E$ is pseudo-effective if and only if
    \[
    E \cdot C \geq 0
    \]
    for all irreducible movable curves $C$ on $V$.
    \end{thm}
    \begin{proof}
        See \cite[Section 2]{BDPP}.
    \end{proof}

    \begin{cor}
        If $E$ is numerically equivalent to a pseudo-effective divisor, then $E$ is pseudo-effective.
    \end{cor}
    \begin{proof}
    Suppose that $D \equiv_n E$ and $D$ is pseudo-effective. Then, for any irreducible movable curve $C$,
    \[
    E \cdot C = D \cdot C \geq 0.
    \]
    Therefore, because of Theorem~\ref{movablecurve}, $E$ is pseudo-effective.
    \end{proof}

    \subsection{Abelian Varieties}
    We refer \cite{La2} to the reader for details on abelian varieties.

    \begin{df}
        Let $A$ be an abelian variety and let $D$ be a divisor on $A$. We say that \emph{$D$ is symmetric} if
        $[-1]^*D$ is linearly equivalent to $D$.
    \end{df}

    \begin{df}\label{BF}
        Let $A$ be an abelian variety defined over a number field $k$ and let $D$ be a symmetric ample divisor on $A$. We define \emph{the canonical height function on $A\bigl(\overline{k}\bigr)$ associated with $D$} to be
        \[
        \widehat{h}_{A,D}(a) := \lim_{m\rightarrow \infty} \dfrac{1}{2^m}h_{A,D}\bigl( [2^m]a \bigr)
        \]
        where $h_{A,D}$ is a Weil height function associated with $D$.
    \end{df}

    \begin{prop}\label{canonicalheight}
    Let $A$ be an abelian variety defined over a number field $k$ and let $D$ be a symmetric ample divisor on $A$. Then,
        \[
        \widehat{h}_{A,D}(a) = h_{A,D}(a) + O(1).
        \]
    \end{prop}
    \begin{proof}
        See \cite[Theorem B.5.6]{SH}.
    \end{proof}

    \begin{lem}\label{abelianpullback}
        Let $A$ be an abelian variety defined over a number field $k$ and let $\widehat{h}_D$ be a canonical height function on $A$ associated with a symmetric ample divisor $D$. Then we have
        \[
        \widehat{h}_D(a) = 0 \quad \text{if and only if} \quad a~\text{is of finite order}.
        \]
    \end{lem}
    \begin{proof}
        See \cite[Proposition~B.4.2, Theorem B.5.1]{SH}.
    \end{proof}

    \begin{lem}\label{finite}
        Let $A$ be an abelian variety defined over a number field $k$ and let $D$ be an ample divisor on $A$.
        Then, $\widehat{h}_D$ is a quadratic form and hence we have a corresponding bilinear form:
        \[
        \langle \cdot ,\cdot \rangle_D: A\bigl(\overline{k}\bigr) \times A\bigl(\overline{k}\bigr) \rightarrow \mathbb{R}, \quad
        \langle a,b \rangle_D = \dfrac{1}{2} \left( \widehat{h}_{A,D}(a+b) - \widehat{h}_{A,D}(a)-\widehat{h}_{A,D}(b)  \right).
        \]
        Moreover,
        \[
        \langle a , b \rangle_D =0
        \]
        for all $b\in A$ if and only if $a$ is of finite order.
    \end{lem}
    \begin{proof}
        See \cite[Proposition B.5.3]{SH} to check $\widetilde{h}_D$ is a quadratic form.

        Let $a$ be of order $m$. Then, $m\langle a , b \rangle_D = \langle ma , b \rangle_D = \langle O , b \rangle_D = 0$. On the other hand,
        if $a$ is not finite order, then
        \[
        \langle a , a \rangle_D =  \widehat{h}_D(a) \neq 0.
        \]
    \end{proof}

    \begin{prop}\label{proptrans}
    Let $A$ be an abelian variety and let $D$ be a divisor on $A$. Let $t_a:A\to A$ is a translation endomorphism by $a\in A$:
    \[
    t_a(b) = a+b\quad \text{for all }b\in A.
    \]
    Then, there is a group homomorphism
    \[
    \Phi_D : A \rightarrow \Pic(A), \quad a \mapsto t^*_a D - D
    \]
    satisfying
    \begin{enumerate}
        \item The image of $\Phi_D$ lies in $\Pic^0(A)$.
        \item If the divisor $D$ is ample, then $\Phi_D$ is surjective and has finite kernel.
    \end{enumerate}
    \end{prop}
    \begin{proof}
        See \cite[Theorem A.7.3.1]{SH}.
    \end{proof}

    \begin{prop}
        Let $A,B$ be abelian varieties defined over a number field $k$ and let $\phi:B\rightarrow A$ be a morphism. Then,
        \[
        \widehat{h}_{B,\phi^*D}(b) = \widehat{h}_{A,D}\bigl(\phi(b)\bigr) - \widehat{h}_{A,D}\bigl(\phi(O)\bigr).
        \]
    \end{prop}
    \begin{proof}
        See \cite[Theorem B.5.6]{SH}.
    \end{proof}

    \begin{lem}\label{squaretheorem}
        Let $A$ be an abelian variety, let $D$ be a divisor on $A$ and let $t_a$ be a translation map by $a$.
        Then, we have
        \[
        t_{a+b}^*D + D \sim t_{a}^*D + t_b^*D\quad
        \text{for all }a,b \in A.\]
    \end{lem}
    \begin{proof}
        \cite[Theorem A.7.2.9]{SH}
    \end{proof}

    \subsection{Albanese Varieties}
    We refer \cite[II,\S2]{La2} for details on Albanese varieties.

    \begin{df}
        Let $V$ be a projective variety. Then, we define \emph{the Albanese variety of $V$} to be
        \[
        \Alb(V) :=H^0(V,\Omega^1_V)^* / H^1(V,\mathbb{Z})
        \]
        where $H^0(V,\Omega^1_V)^*$ is the dual of $H^0(V,\Omega^1_V)$ and $\Omega^1_V$ is the set of $1$-forms on $V$.
    \end{df}

    \begin{rem}
        $\Alba(V)$ is an initial object: if there is any morphism $\phi : V \rightarrow A$ where $A$ is an abelian variety, then
        there is a morphism $\psi : \Alba(V) \rightarrow A$ such that $\phi = \psi \circ \pi$ where $\pi : V \rightarrow \Alb(V)$ is the universal map.
    \end{rem}

    \begin{prop}\label{albiso}
        Let $V$ be a projective variety, let $A=\Alb(V)$ and let $\pi:V \rightarrow A$ be the universal map from $V$ to its Albanese variety. Then, the pullback map $\pi^*: \Pic^0(A) \rightarrow \Pic^0(V)$ is an isomorphism.
    \end{prop}
    \begin{proof}
    See \cite[Section IV.4 and VI1, Theorem 1]{La2}.
    \end{proof}

    \begin{thm}\label{gen}
    Let $V$ be a projective variety, let $A=\Alb(V)$ and let $\pi:V \rightarrow A$ be the universal map from $V$ to its Albanese variety. Then there exists a positive integer $N$ such that
    \[
    F : V \times \cdots \times V \to A, \quad F(P_1, \cdots, P_N) = \sum_{i=1}^N \pi(P_i)
    \]
    is a surjective map.
    \end{thm}
    \begin{proof}
    See \cite[Theorem~11]{La2}.
    \end{proof}

\section{The kernel of the Weil height machine}

    Serre \cite{Se} showed that the kernel of the Weil height machine is the torsion subgroup of the Picard Group. But, his proof is quite simple so that we can't get further information. In this section, we introduce another proof using Albanese varieties, which helps studying properties of Weil height functions associated with algebraically trivial or numerically trivial divisors.

    \begin{lem}\label{algzerobounded}
        Let $V$ be a projective variety defined over a number field $k$ and let $E$ be a divisor on $V$. Suppose that $E$ is algebraically trivial. Then $|h_E(P)|$ is bounded if and only if $E$ is of finite order in $\Pic(V)$.
    \end{lem}
    \begin{proof} One direction is clear: if $E$ is of order $m$, then
    \[
    h_{E}(P) = \dfrac{1}{m} h_{mE}(P) = \dfrac{1}{m} h_{O}(P) + O(1) = O(1).
    \]

    For the other direction, let $A = \Alb(V)$. Then, we have the universal map $\pi : V \rightarrow A$ such that the group homomorphism $\pi^* : \Pic^0(A) \rightarrow \Pic^0(V)$ is an isomorphism by Proposition~\ref{albiso}. Moreover, by definition of algebraically trivial divisors (Definition~\ref{at}), an algebraically trivial divisor $E$ is connected to the zero divisor. Thus, $E \in \Pic^0(V)$ and hence there is a divisor $E_0 \in \Pic^0(A)$ such that $E = \pi^*E_0$ by Proposition~\ref{albiso}.

    Choose a symmetric ample divisor $E_S$ on $A$. Then,
    we have a translation map
    \[
    \Phi_{E_S} : A \rightarrow \Pic^0(A)
    \]
    which is surjective (Proposition~\ref{proptrans}). Hence, there is a point $a_0 \in A$ such that $E_0 = \Phi_{E_S}(a_0) = t^*_{a_0} E_S - E_S$.

    Furthermore, we get
    \begin{equation}\label{bb}\tag{A}
    \begin{array}{rcll}
    h_{V,E}(P) &=&  h_{V,\pi^*E_0}(P) + O(1) \\
               &=&  h_{A,E_0}(a) + O(1)  & \bigl( a = \pi(P) \bigr)\\
               &=&  \widehat{h}_{A,E_0}(a) + O(1) &(\because Proposition~\ref{canonicalheight})\\
               &=&  \widehat{h}_{A,t^*_{a_0} E_S}(a) - \widehat{h}_{A,E_S}(a)+ O(1) & \bigl(\because E_0 = t_{a_0}^*E_S - E_S \bigr)\\
               &=&  \widehat{h}_{A,E_S}(t_{a_0}(a)) - \widehat{h}_{A,E_S}(t_{a_0}(O)) - \widehat{h}_{A,E_S}(a)+ O(1) & (\because Lemma~\ref{abelianpullback})\\
               &=&  \widehat{h}_{A,E_S}(a+a_0) - \widehat{h}_{A,E_S}(a_0) - \widehat{h}_{A,E_S}(a)+ O(1) \\
               &=&  2\langle a, a_0 \rangle_{E_S} + O(1).
    \end{array}
    \end{equation}
    Thus, the boundedness of $h_{V,E}$ guarantees that $\langle \pi(P), a_0 \rangle$ is bounded for all $P\in V$.

    By Theorem~\ref{gen}, there exists a positive integer $N$ such that
    \[
    F : \prod^N V  \rightarrow \Alb(V) , \quad (P_1, \cdots, P_N) \mapsto \sum_{i=1}^N \pi(P_i)
    \]
    is a surjective map. So, $\langle \cdot , a_0 \rangle_{E_S}$ is bounded on $A$:
    \[
    |\langle a, a_0 \rangle_{E_S}| \leq \sum_{i=1}^N |\langle \pi(P_i), a_0 \rangle_{E_S}| \leq O(1)\quad \text{for all }a\in A.
    \]
    Since it is a bilinear map, the boundedness means trivial:  $\langle [m]a, a_0 \rangle_{E_S} = m\langle a, a_0 \rangle_{E_S}$ is bounded for all $m$ and hence $\langle a, a_0 \rangle_{E_S}=0$. So, $a_0$ is of finite order by Lemma~\ref{finite}.

    Therefore, $E_0 = \Phi_{E_S}(a_0)$ is also of finite order and hence $E=\pi^*E_0$ is of finite order because $\Phi_{E_S}, \pi^*$ are group homomorphisms.
    \end{proof}

    \begin{cor}\label{numzerobounded}
        Let $V$ be a projective variety defined over a number field $k$ and let $E$ be a divisor on $V$ with $E$ numerically trivial. Then, $|h_E(P)|$ is bounded if and only if $E$ is of finite order in $\Pic(V)$.
    \end{cor}
    \begin{proof}
        One direction is clear: if $E$ is of finite order then $|h_E(P)|$ is bounded.

        Suppose that $E$ is numerically trivial and $|h_{E}(P)|$ is bounded. Then, there is a positive integer $m$ such that $mE\equiv 0$ (Remark~\ref{NS}). Because $|h_{mE}(P)|$ is still bounded, $mE$ is of finite order and hence so is $E$.
    \end{proof}

    \begin{thm}
        The kernel of the Weil height machine
        \[
        h_V : \Pic(V) \rightarrow \Ht(V\bigl(\overline{k}\bigr))
        \]
        is the torsion subgroup of $\Pic(V)$
    \end{thm}
    \begin{proof}
     One direction is clear: a torsion element of $\Pic(V)$ generates a bounded Weil height function.

     Suppose that $|h_E(P)|$ is bounded. It is enough to show that $E\equiv_n 0$ because of Corollary~\ref{numzerobounded}. Take an ample divisor $D$ on $V$ and an irreducible curve $C$ on $V$. By assumption, $|h_E(P)|$ is bounded on $C \subset V$ and hence
     \[
     \dfrac{E \cdot C}{D\cdot C} = \dfrac{\deg \iota^*E}{\deg \iota^*D} = \lim_{h_{\iota^*D}(Q) \rightarrow \infty} \dfrac{h_{\iota^*E} (Q)}{h_{\iota^*D}(Q)}=\lim_{h_{D}(P) \rightarrow \infty} \dfrac{h_{E} (P)}{h_{D}(P)} = 0
     \]
     where $\iota : C \rightarrow V$ is the closed embedding of $C$ into $V$ and $Q$ is a point on $C$. By Kleiman's Criterion (Theorem~\ref{KC}), we get $D\cdot C >0$ and hence $E \cdot C = 0$. Therefore, $E\equiv_n 0$.
    \end{proof}

    \begin{cor}
        \[
        h_{D_1} \sim h_{D_2} \quad \text{if and only if} \quad D_1-D_2 \sim E ~\text{for some}~ E \in \bigl(\Pic(V) \bigr)_{tor}.
        \]
    \end{cor}

   \begin{cor}
        Let $V$ be a projective variety defined over a number field $k$ and let $E$ be a divisor on~$V$. Suppose that $E$ is algebraically trivial. Then, either
        \begin{enumerate}
            \item $|h_E|$ is bounded, or
            \item $h_E$ is not bounded above nor below.
        \end{enumerate}
    \end{cor}
    \begin{proof}
        Suppose that $E \equiv 0$ and $|h_E|$ is not bounded. Then, by (\ref{bb}) in the proof of Theorem~3.1, we have
        \[
        h_E(P) = 2\langle a,a_0 \rangle_{E_S} +O(1)
        \]
        where $a = \pi(P)$ and $E = \pi^*t_{a_0}^* E_S$. Recall that $a_0$ is a point of infinite order if $E$ is not torsion.
        By Theorem~\ref{gen}, there exists a positive integer $N$ such that
        \[
        F : \prod^N V  \rightarrow \Alb(V) , \quad (P_1, \cdots, P_N) \mapsto \sum_{i=1}^N \pi(P_i)
        \]
        is generically surjective. So, there is a sequence of $N$-tuples of points $T_m=(P_{1,m}\cdots, P_{N,m})$ such that
        \[
        F(T_m) = [m]a_0.
        \]

       If $\langle \pi(P_{i,m}) ,a_0 \rangle_{E_S} < \dfrac{m}{N} \widehat{h}_{E_S}(a_0)$ for all $m>0$ and $i= 1, \cdots, N$. Then we get
       \begin{eqnarray*}
        m\widehat{h}_{E_S}(a_0) &=&  \sum_{i=1^N} \langle \pi(P_{i,m}) ,a_0 \rangle_{E_S}  \\
                                &\leq &  \sum_{i=1^N}  \langle \pi(P_{i,m}) ,a_0 \rangle_{E_S}  \\
                                &< &  m \widehat{h}_{E_S}(a_0)
        \end{eqnarray*}
        which is a contradiction. Hence, there is an $i_m$ for each $m>0$ satisfying
        \[
        \langle \pi(P_{i_m, m}) ,a_0 \rangle_{E_S} \geq  \dfrac{m}{N} \widehat{h}_{E_S}(a_0).
        \]
        Therefore, we get
           \[
            \lim_{m \rightarrow \infty} h_E(P_{i_m,m}) \geq  \lim_{m \rightarrow \infty} \dfrac{m}{N}\widehat{H}_{E_S} (a_0) + O(1) \geq \infty.
            \]

        Similarly, there is a $j_{m}$ for each $(-m)<0$ satisfying
        \[
        \langle \pi(P_{j_{m}, -m}) ,a_0 \rangle_{E_S} \leq  \dfrac{-m}{N} \widehat{h}_{E_S}(a_0)
        \]
        and hence we get
           \[
            \lim_{m \rightarrow \infty} h_E(P_{j_{m},-m}) \leq  \lim_{m \rightarrow \infty} \dfrac{-m}{N}\widehat{H}_{E_S} (a_0) + O(1) \leq -\infty.
            \]
    \end{proof}

\section{The numerical and algebraic equivalences of height functions}

    In this section, we define algebraic and numerical equivalence relations of Weil height functions. Though the algebraic and the numerical equivalence relations of divisors are strictly different, but the algebraic and the numerical equivalence relations of Weil height functions are same. So, we can have another description of the numerical N\'{e}ron-Severi group.
    Also, we can see that the numerical equivalence relation of Weil height functions shows numerical information of corresponding divisors.

    \begin{df}
        Let $V$ be a projective variety defined over a number field $k$, let $D_1,D_2$ be divisors on a projective variety $V$ and let $h_i$ be a Weil height function associated with $D_i$. We say
        \emph{$h_1$ is algebraically equivalent to $h_2$} if
        \[
        D_1 - D_2 \equiv E \quad \text{for some torsion element }E \in \Pic(V)
        \]
        and write $h_1 \equiv h_2$.
        Also, we say \emph{$h_1$ is numerically equivalent to $h_2$} if $D_1 \equiv_n D_2$ and write $h_1 \equiv_n h_2$.
    \end{df}

    \begin{prop}\label{proposition_numeric}
        Let $h_1, h_2$ be Weil height functions on $V$ associated with $D_1$ and $D_2$ respectively. Then
        \begin{enumerate}
        \item $h_1 \equiv h_2 \quad \text{if} \quad h_1 \sim h_2.$
        \item $h_1 \equiv_n h_2 \quad \text{if} \quad h_1 \sim h_2.$
        \item $h_1 \equiv_n h_2 \quad \text{if and only if} \quad h_1 \equiv h_2.$
        \item If $h_1 \equiv_n h_2$, then
        \[
        \lim_{n \to \infty } \dfrac{h_1(P_n)}{h_D(P_n)} = \lim_{n \to \infty } \dfrac{h_2(P_n)}{h_D(P_n)}
        \]
        for any sequence $\{P_n\} \subset V$ and any ample divisor $D$ on $V$.
        \end{enumerate}
    \end{prop}
    \begin{proof}
        \begin{enumerate}
        \item If $h_1 \sim h_2$, then $D_1 - D_2 \sim E$ for some torsion element $E$ in $\Pic(V)$ so that
        $D_1 - D_2 \equiv E$.\\

         \item Suppose $h_1 \sim h_2$. Then, $h_{D_1-D_2}$ is a bounded height so that $D_1-D_2$ is a torsion element of $\Pic(V)$ of finite order $m$. Therefore, for any curve $C$ on $V$,
        \[
        m(D_1-D_2) \cdot C = O \cdot C = 0
        \]
        and hence $D_1 \cdot C = D_2 \cdot C$. \\

         \item One direction is clear: $D$ is numerically equivalent to $E$ if $D$ is algebraically equivalent to $E$. Moreover, if $E$ is of finite order, then $E$ is numerically trivial. 
             
             For the other direction, consider the following diagram;
        \[
        \xymatrix{
        \Pic(V)/\bigl(\Pic(V) \bigr)_{tor}  \ar[d]_{\pi_a} \ar[r]^-\sim  &\Ht(V\bigl(\overline{k}\bigr)) \ar[d]_{\pi_a}\\
        \NS(V)/\bigl(\NS(V) \bigr)_{tor}  \ar[d]_{\pi_n} \ar[r]^-\sim  & \Ht(V\bigl(\overline{k}\bigr)) /\equiv \ar[d]_{\pi_n} \\
        N^1(V)/\bigl(N^1(V) \bigr)_{tor}  \ar[r]^-\sim  & \Ht(V\bigl(\overline{k}\bigr)) /\equiv_n \\
        }
        \]
        We know that all horizontal group homomorphisms are isomorphisms. Furthermore,
        \[
        \NS(V)/\bigl(\NS(V) \bigr)_{tor}  = N^1(V) \simeq N^1(V)/\bigl(N^1(V) \bigr)_{tor}
        \]
        so that corresponding height groups are also isomorphic:
        \[
         \Ht(V\bigl(\overline{k}\bigr))/\equiv_n  ~\simeq~  \Ht(V\bigl(\overline{k}\bigr))/\equiv .
        \]
        \hspace{0.3mm}

        \item Suppose that $h_1 \equiv_n h_2$ and $D$ is an ample divisor. Then,
        there is a nonzero integer $m$ satisfying $mE = mD_1 - mD_2 \equiv 0$. By the algebraic equivalence property of the Weil height machine (Theorem~\ref{WHM},(6)), we get
        \[
        \lim_{{\tiny h_D(P) \rightarrow \infty} } \dfrac{h_{mE}(P)}{h_D(P)} = 0.
        \]
        Therefore,
        \begin{equation}\tag{B}\label{DD}
        0 = \lim_{h_D(P) \rightarrow \infty } \dfrac{1}{m}\dfrac{h_E(P)}{h_D(P)}=\lim_{h_D(P) \rightarrow \infty } \left( \dfrac{h_1(P)}{h_D(P)} ~- \dfrac{h_2(P)}{h_D(P)} \right)
        \end{equation}
        and hence both $\displaystyle \lim_{n \to \infty } \dfrac{h_1(P_n)}{h_D(P_n)}$ and $\displaystyle\lim_{n \to \infty } \dfrac{h_2(P_n)}{h_D(P_n)}$ diverge or converge to the same number.
        (Note that the limit on {\rm (\ref{DD})} is defined but $\displaystyle \lim_{h_D(P) \rightarrow \infty } \dfrac{h_1(P)}{h_D(P)}$ may not be defined so that we will treat $\limsup$ and $\liminf$ in Section ~5.)
        \end{enumerate}
    \end{proof}

\section{ample divisors and the fractional limit of Weil heights}

    In this section, we will prove the main theorem. Proposition~\ref{proposition_numeric}~(4) shows that
    \[
    \liminf_{h_D(P) \rightarrow \infty } \dfrac{h_1(P)}{h_D(P)} = \liminf_{h_D(P) \rightarrow \infty } \dfrac{h_2(P)}{h_D(P)},
    \quad \limsup_{h_D(P) \rightarrow \infty } \dfrac{h_1(P)}{h_D(P)} = \limsup_{h_D(P) \rightarrow \infty } \dfrac{h_2(P)}{h_D(P)}
    \]
    if $h_q \equiv_n h_2$. We will show that they are actually finite numbers having numerical information of corresponding divisors. We start by showing the choice of an ample divisor $D$ does not affect to the desired result.

    \begin{lem}\label{changelimit}
        Let $V$ be a projective variety defined over a number field $k$ and let $D_1,D_2$ be ample divisors on $V$. Then, for any sequence $\{P_n\} \subset V\bigl( \overline{k}\bigr)$,
        \[
        \lim_{n \rightarrow \infty} h_{D_1}(P_n) = \infty \quad \text{if and only if} \quad \lim_{n \rightarrow \infty} h_{D_2}(P_n) = \infty.
        \]
    \end{lem}
    \begin{proof}
    Suppose that there is a sequence of points $\{P_n\}$ such that
    \[
       \lim_{n \rightarrow \infty} h_{D_1}(P_n) = \infty \quad \text{and} \quad \limsup_{n \rightarrow \infty} h_{D_2}(P_n) < M
    \]
    for some positive number $M$. Find $m$ such that $mD_2 - D_1$ is ample (Lemma~\ref{ampledifference}) and get
    \[
    \liminf_{h_{D_1}(P) \rightarrow \infty} \dfrac{h_{mD_2 - D_1}(P)}{h_{D_1}(P)} \leq \limsup_{n \rightarrow \infty} \dfrac{h_{mD_2 - D_1}(P_n)}{h_{D_1}(P_n)} = -1.
    \]
    However, the Weil height function associated with an ample divisor $mD_2 - D_1$ is bounded below (Theorem~\ref{WHM}~(7)) so that
    \[
    \liminf_{n \rightarrow \infty} \dfrac{h_{mD_2 - D_1}(P_n)}{h_{D_1}(P_n)} \geq 0,
    \]
    which is a contradiction.
    \end{proof}

    \begin{prop}\label{amplelimit}
        Let $V$ be a projective variety defined over a number field $k$ and let $D_1,D_2$ be ample divisors on $V$. Then,
         \[
        0< \liminf_{h_{D_2}(P) \rightarrow \infty} \dfrac{h_{D_1}(P)}{h_{D_2}(P)} \leq \limsup_{h_{D_2}(P) \rightarrow \infty} \dfrac{h_{D_1}(P)}{h_{D_2}(P)} <\infty.
        \]
    \end{prop}
    \begin{proof} Since $D_1$ and $D_2$ are ample, there is a constant $m_1 >0$ such that $m_1D_1 - D_2$ is ample again by Lemma~\ref{ampledifference}.
    Weil height functions associated with ample divisors are bounded below (Theorem~\ref{WHM}~(7)) and hence
    $h_{D_1}(P) > \dfrac{1}{m_1} \left(  h_{D_2}(P) - O(1) \right)$. Therefore,
         \begin{equation}\label{liminf}\tag{C}
            \liminf_{h_{D_2}(P) \rightarrow \infty} \dfrac{h_{D_1}(P)}{h_{D_2}(P)}
            = \liminf_{h_{D_2}(P) \rightarrow \infty} \dfrac{1}{m_1} \dfrac{h_{D_2}(P) - O(1)}{h_{D_2}(P)} = \dfrac{1}{m_1} >0.
        \end{equation}

    Because of the Lemma~\ref{changelimit}, we can change the limit and hence we get
        \begin{equation}\label{limsup}\tag{D}
        \limsup_{h_{D_2}(P) \rightarrow \infty} \dfrac{h_{D_1}(P)}{h_{D_2}(P)}
        =\dfrac{1}{\displaystyle\liminf_{h_{D_2}(P) \rightarrow \infty} \dfrac{h_{D_2}(P)}{h_{D_1}(P)}}
         =\dfrac{1}{\displaystyle\liminf_{h_{D_1}(P) \rightarrow \infty} \dfrac{h_{D_2}(P)}{h_{D_1}(P)}} < \infty.
        \end{equation}
    Finally, combine (\ref{liminf}) and (\ref{limsup}) to get the desired result.
    \end{proof}

    \begin{cor}
        Let $V$ be a projective variety defined over a number field $k$ and let $D,E$ be divisors on $V$ with $D$ ample. Then, we have
        \[
        \limsup_{h_{D}(P) \rightarrow \infty}
        \left| \dfrac{h_{E}(P)}{h_{D}(P)} \right| <\infty.
        \]
    \end{cor}
    \begin{proof}
    Suppose that $E \sim D_1 - D_2$ where $D_i$ are ample divisors. Then, by Proposition~\ref{amplelimit}, we get
    \[
    \limsup_{h_{D}(P) \rightarrow \infty} \dfrac{h_{E}(P)}{h_{D}(P)}
    \leq \limsup_{h_{D}(P) \rightarrow \infty} \dfrac{h_{D_1}(P)}{h_{D}(P)} -
    \liminf_{h_{D}(P) \rightarrow \infty} \dfrac{h_{D_2}(P)}{h_{D}(P)}
    < \limsup_{h_{D}(P) \rightarrow \infty} \dfrac{h_{D_1}(P)}{h_{D}(P)} <\infty.
    \]

    The other inequality is easily gained from the first inequality:
    \[
    -\liminf_{h_{D}(P) \rightarrow \infty} \dfrac{h_{E}(P)}{h_{D}(P)} = \limsup_{h_{D}(P) \rightarrow \infty} \dfrac{h_{-E}(P)}{h_{D}(P)} <\infty.
    \]
    \end{proof}

    \begin{prop} Let $V$ be a projective variety defined over a number field $k$ and let $E$ be a divisor on $V$. Then, the followings are equivalent:
        \begin{enumerate}
            \item $\displaystyle \lim_{h_D(P)\rightarrow \infty} \dfrac{h_E(P)}{h_D(P)} = 0 $ for all ample divisor $D$.
            \item $\displaystyle \lim_{h_D(P)\rightarrow \infty} \dfrac{h_E(P)}{h_D(P)} = 0 $ for some ample divisor $D$.
            \item $E \equiv_n 0$.
        \end{enumerate}
    \end{prop}
    \begin{proof}
        (1) $\Rightarrow$ (2) is clear and
(3) $\Rightarrow$ (1) is shown in Proposition~\ref{proposition_numeric}~(5).

        Show ($2$) $\Rightarrow$ ($1$) first. Suppose that ($2$) is true for an ample divisor $D_1$. Then, for any other ample divisor $D_2$,
        we get
        \[
        \limsup_{h_{D_2}(P) \rightarrow \infty} \left| \dfrac{h_{D_1}(P)}{h_{D_2}(P)} \right| = C <\infty
        \]
        by Proposition~\ref{amplelimit}. Furthermore, we can change the limit because of Lemma~\ref{changelimit}. Therefore, we get
        \[
        \lim_{h_{D_2} (P)\rightarrow \infty} \left| \dfrac{h_E(P)}{h_{D_2}(P)} \right|
        \leq \limsup_{h_{D_2}(P) \rightarrow \infty} \left| \dfrac{h_{D_1}(P)}{h_{D_2}(P)} \right| \cdot \lim_{h_{D_1} (P)\rightarrow \infty} \left| \dfrac{h_E(P)}{h_{D_1}(P)} \right| = C\cdot 0
        \]
        and hence ($1$) holds for $D_2$.

        Now we prove ($1$) $\Rightarrow$ ($3$): suppose that $E$ is not numerically trivial. Then, there is an irreducible curve $C$ such that $E\cdot C \neq 0$. Pick an ample divisor $D$. Then, by Nakai-Moishezon Criterion (Theorem~\ref{NMC}), we get $D\cdot C >0$. Thus,
        \[
        \lim_{\substack{P\in C \\ h_D(P) \rightarrow \infty} } \dfrac{h_E(P)}{h_D(P)} = \dfrac{E\cdot C }{D \cdot C} \neq 0
        \]
        so ($1$) fails. Hence, ($1$) holds only if ($3$) holds.
    \end{proof}

    Now we can prove the main theorem of this paper.
            \begin{thmA}
        Let $V$ be a projective variety defined over a number field $k$, let $U$ be a subset of $V$ and let $E,D$ be divisors on $V$ with $D$ ample. We define \emph{the fractional limit of $h_E$ with respect to $h_D$ on $U$} to be
        \[
        \Flim_D(E,U) := \liminf_{\substack{P \in U \\  h_D(P) \rightarrow \infty}}\dfrac{h_E(P)}{h_D(P)}.
        \]
        The fractional limit satisfies the following properties:
        \begin{enumerate}
        \item $E$ is ample if and only if  $\Flim_D(E,V) >0.$
        \item $E$ is numerically effective if and only if  $\Flim_D(E,V)  \geq 0.$
        \item $E$ is effective only if
        $\Flim_D(E,U)  \geq 0$
        for some dense open set $U \subset V$.
        \item $E$ is pseudo-effective only if
        $\Flim_D(E,U)  \geq 0$
        for some $U$ which is an infinite intersection of open subsets of $V$.
        \item $E$ is pseudo-effective if
        $\Flim_D(E,U)  \geq 0$
        for some dense open set $U \subset V$.
        \item Suppose that any pseudo-effective divisor on $V$ is linearly equivalent to a sum of an effective divisor and a numerically effective divisor. Then, $E$ is pseudo-effective if only if
        $\Flim_D(E,U) \geq 0$
        for some dense open set $U \subset V$.
        \end{enumerate}
    \end{thmA}

    \begin{rem}
    The condition on (6) is a necessary condition for various versions of Zariski decomposition. (For details of Zariski decomposition, see \cite{BDPP, H, La}.) For example, a projective surface $S$ allows Zariski decomposition so that a divisor $E$ on $S$ is pseudo-effective if only if $\Flim_D(E,U) \geq 0$ for some dense open set $U \subset V$.
    \end{rem}

    \begin{proof}
    \begin{enumerate}
    \item The `only if' part is proved by Proposition~\ref{amplelimit}. So, it's enough to show the `if' part. Suppose that
        \[
        \liminf_{h_D(P) \rightarrow \infty}\dfrac{h_E(P)}{h_D(P)} = \alpha >0.
        \]
        Then, for any effective cycle $C \in \EFC(V)$, we may assume that
        \[
        C = \sum_{j=1}^n \beta_j C_j
        \]
        where $\beta_j$ are positive real numbers and $C_j$ are irreducible curves. Then, for any curve $C_j$,
        \[
        \alpha \leq \liminf_{\substack{h_D(P) \rightarrow \infty \\ P \in C_j} }\dfrac{h_E(P)}{h_D(P)} = \dfrac{E\cdot C_j}{D \cdot C_j}
        \]
        so that
        \[
        E\cdot C = \sum_{j=1}^n \beta_j \bigl( E\cdot C_j\bigr)\geq \sum_{j=1}^n \beta_j \bigl( \alpha D\cdot C_j\bigr) \geq \alpha \sum_{j=1}^n \beta_j >0.
        \]
        Moreover, for any cycle $C' \in \overline{\EFC}(V)\setminus \EFC(V)$, there a sequence $C_k$ of nonzero effective cycles converging to $C'$. Suppose that $C_k = \displaystyle \sum_{j=1}^{j_k} \beta_{j,k} C_{j,k}$ where $C_{j,k}$ are irreducible curves and $\beta_{j,k}$ are positive real numbers. Since
        $C_k \cdot E >0$, we get $C' \cdot E \geq 0$. Moreover, if $C' \cdot E =0$, then
        \[
        0 = \lim_{k\rightarrow \infty} C_k \cdot E \geq \alpha \sum_{j=1}^{j_k} \beta_{j,k} \geq 0
        \]
        so that $\displaystyle \lim_{k\rightarrow \infty} \sum_{j=1}^{j_k} {\beta_{j,k}}=0$ and hence $C' = \lim C_k = 0$. Therefore, if $C'$ is nonzero,
        then $C' \cdot E >0$ and hence $E$ is ample by Kleiman's Criterion.\\

        \item Let $E$ be a numerically effective divisor. Then, by Kleiman's Criterion, $E_m = E + \dfrac{1}{m}D$ is ample for all $m>0$ and all ample divisor $D$. Thus, by (1), we have
        \[
        \liminf_{h_D(P) \rightarrow \infty}\dfrac{h_{E_m} (P)}{h_D(P)}>0.
        \]
        Therefore, $\Flim_D \geq 0$ because the following inequality holds for all$m>0$:
        \[
        \liminf_{h_D(P) \rightarrow \infty}\dfrac{h_{E} (P)}{h_D(P)} \geq 
        \liminf_{h_D(P) \rightarrow \infty}\dfrac{h_{E_m} (P)}{h_D(P)} - \dfrac{1}{m} \limsup_{h_D(P) \rightarrow \infty}\dfrac{h_{D} (P)}{h_D(P)} \geq -\dfrac{1}{m}.
        \]
        
        Conversely, if $E$ is not numerically effective, then there is a curve $C$ such that $E \cdot C <0$. Therefore,
        \[
        \liminf_{\substack{P\in V \\ h_D(P) \rightarrow \infty} }\dfrac{h_E(P)}{h_D(P)} \leq \liminf_{ \substack{P\in C \\ h_D(P) \rightarrow \infty} }\dfrac{h_E(P)}{h_D(P)} = \dfrac{E\cdot C}{D\cdot C} <0.
        \]

        \item If $E$ is effective, then $h_E$ is bounded below on $V\setminus |E|$, where $|E|$ is the base locus of $E$.
        Thus, we get
        \[
        \liminf_{ \substack{P \in V\setminus |E| \\  h_D(P) \rightarrow \infty}}\dfrac{h_E(P)}{h_D(P)}
        \geq  \liminf_{ \substack{P \in V\setminus |E| \\  h_D(P) \rightarrow \infty}}\dfrac{C}{h_D(P)} = 0.
        \]

        \item Suppose that $E$ is pseudo-effective. Then, there is an ample divisor $D_0$ such that $E+ \epsilon D_0$ is effective for any $\epsilon >0$. Construct a set $\mathbf{B}$, an infinite union of a closed subsets of $V$:
        \[
        \mathbf{B}(E) = \bigcup_{n=1}^\infty \left| E+ \dfrac{1}{n}D_0 \right|
        \]
        where $B_n = \left| E+ \dfrac{1}{n}D_0 \right|$ is the base locus of $ E+ \dfrac{1}{n}D_0 $. Since a Weil height function associated with an effective divisor is bounded below outside of the base locus (Theorem~\ref{WHM}~(5)), we get
        \[
        \liminf_{ \substack{P \in V\setminus B_n \\  h_{D_0}(P) \rightarrow \infty} }\dfrac{h_E(P)}{h_{D_0}(P)} \geq 0.
        \]
        Therefore, the following inequality holds on a set $U = V \setminus \mathbf{B}(E)$ which is an infinite intersection of open subsets of $V$ for all $n$:
        \begin{eqnarray*}
        \liminf_{\substack{P \in U \\  h_{D_0}(P) \rightarrow \infty}}\dfrac{h_{E} (P)}{h_{D_0}(P)}
        &= &\liminf_{\substack{P \in U \\  h_{D_0}(P) \rightarrow \infty}} \left( \dfrac{h_{E+\frac{1}{n}{D_0}} (P)}{h_{D_0}(P)}
               - \dfrac{\frac{1}{n}h_{{D_0}} (P)}{h_{D_0}(P)}
               \right) \\
        &\geq &\liminf_{\substack{P \in U \\  h_{D_0}(P) \rightarrow \infty}} \dfrac{h_{E+\frac{1}{n}{D_0}} (P)}{h_{D_0}(P)} - \dfrac{1}{n} \\
        &\geq & -\dfrac{1}{n}
        \end{eqnarray*}
        and hence
        \[
        \liminf_{\substack{P \in U \\  h_{D_0}(P) \rightarrow \infty}}\dfrac{h_{E} (P)}{h_{D_0}(P)} \geq 0.
        \]
        Finally, for any ample divisor $D$, we get
        \[
        \liminf_{\substack{P \in U \\  h_D(P) \rightarrow \infty}}\dfrac{h_{E} (P)}{h_D(P)} \geq
         \liminf_{\substack{P \in U \\  h_{D_0}(P) \rightarrow \infty}}\dfrac{h_{E} (P)}{h_{D_0}(P)}
         \times
         \liminf_{\substack{P \in U \\  h_D(P) \rightarrow \infty}}\dfrac{h_{{D_0}} (P)}{h_D(P)}
         \geq 0.
        \]

         \item Suppose that $E$ is not pseudo-effective. Then, there is a movable curve $C$ such that $E\cdot C<0$ because of  Theorem~\ref{movablecurve}. Then, for any dense open set $U$ of $V$, there is a curve $C_t$ in the irreducible family $\{C_s ~|~ s \in S\}$ of $C$ such that $C_t \cap U \neq \emptyset$. Moreover, $E\cdot C_t = E \cdot C <0$ and hence
        \[
        \liminf_{\substack{P \in U \\  h_D(P) \rightarrow \infty}}\dfrac{h_{E} (P)}{h_D(P)} \leq  \liminf_{\substack{P \in U \cap C_t \\  h_D(P) \rightarrow \infty}}\dfrac{h_{E} (P)}{h_D(P)} = \dfrac{E\cdot C_t }{R \cdot C_t} < 0.
        \]
        Therefore, if there is a dense open set satisfying the following condition:
        \[
        \liminf_{\substack{P \in U \\  h_D(P) \rightarrow \infty}}\dfrac{h_{E} (P)}{h_D(P)} \geq 0
        \]
        then $E$ is pseudo-effective.

        \item Let $E$ be a pseudo-effective divisor. If $E = E_P + E_N$ for some numerically effective divisor $E_P$
         and some effective divisor $E_N$, then, by (2) and (3), we get
        \[
        \liminf_{\substack{P \in U \\  h_D(P) \rightarrow \infty}}\dfrac{h_{E} (P)}{h_D(P)}
        \geq \liminf_{\substack{P \in U \\  h_D(P) \rightarrow \infty}}\dfrac{h_{E_P} (P)}{h_D(P)}
               + \liminf_{\substack{P \in U \\  h_D(P) \rightarrow \infty}}\dfrac{h_{E_N} (P)}{h_D(P)}
        \geq 0
        \]
        where $U = V \setminus |E_P|$.

        \end{enumerate}
        \end{proof}

        \begin{ex}
            Let $V$ be a projective surface defined over a number field, let $D$ be a ample divisor and let $E$ be a pseudo-effective divisor. Then, $V$ allows the Zariski decomposition and hence
            \[
            E~\text{is pseudo-effective} \quad  \text{if and only if} \quad  \Flim_D(E,U) \geq 0 ~\text{for some open set}~U.
            \]
        \end{ex}

        \begin{thm}
            Let $V,W$ be projective varieties defined over a number field $k$, let $\phi :W \rightarrow V$ be a dominant morphism and let $D_W, D_V$ be ample divisors on $W$ and $V$ respectively. Let $\mu(\phi, D_W, D_V)$ be Silverman's height expansion coefficient of dominant map $\phi$ \cite{S3}:
            \[
            \mu(\phi, D_W, D_V): =\sup_{U\subset W} \liminf_{ \substack{P \in U \\  h_{D_W}(P) \rightarrow \infty} }\dfrac{h_{D_V}\bigl( \phi(P) \bigr)}{h_{D_W} (P)}.
            \]
            Then,
            \[
            \mu(\phi, D_W, D_V) \leq \sup \bigl\{ \alpha ~|~ \phi^* {D_V} - \alpha {D_W} ~ \text{is pseudo-effective.} \bigr\} .
            \]
        \end{thm}
        \begin{proof}
            Suppose $\mu(\phi, D_W, D_V) \geq \alpha $. Then, by definition of $\mu$, there is a open set $U$ such that
            \[
            \liminf_{ \substack{P \in U \\  h_{D_W}(P) \rightarrow \infty} }\dfrac{h_{D_V}\bigl( \phi(P) \bigr)}{h_{D_W} (P)} \geq \alpha
            \]
            Since
            \[
            \lim_{ \substack{P \in U \\  h_{D_W}(P) \rightarrow \infty} }\dfrac{h_{D_W} (P))}{h_{D_W} (P)} = 1,
            \]
            \[
            \liminf_{ \substack{P \in U \\  h_{D_W}(P) \rightarrow \infty} }\dfrac{h_{D_V}\bigl( \phi(P) \bigr)}{h_{D_W} (P)} \geq
            \alpha \cdot \lim_{ \substack{P \in U \\  h_{D_W}(P) \rightarrow \infty} }\dfrac{h_{D_W} (P))}{h_{D_W} (P)}
            \]
            and hence
            \[
            \liminf_{ \substack{P \in U \\  h_{D_W}(P) \rightarrow \infty} }\dfrac{h_{\phi^*{D_V}-\alpha {D_W}}(P)}{h_{{D_W}} (P)} \geq 0.
            \]
            Therefore, $\phi^*{D_V}-\alpha {D_W}$ is pseudo-effective.
        \end{proof}

    \begin{question}
        We want to know Theorem~A~(6) holds in general: let $V(k)$ be a projective variety over a number field $k$ and let $D$ be an ample divisor on $V$. Suppose that we have a pseudo-effective divisor $E$ on $V$. Then is there an open set $U$ of $V$ such that
        \[
        \liminf_{\substack{h_D(P) \rightarrow \infty}\\P \in U }\dfrac{h_E(P)}{h_D(P)} \geq 0 ?
        \]
    \end{question}

\end{document}